\documentclass[12pt]{amsart}
\usepackage{mathrsfs}
\usepackage{amssymb}
\usepackage{amsbsy}
\usepackage{amsmath}
\usepackage{graphicx}
\usepackage{caption}
\usepackage{subfig}
\usepackage[colorlinks,linkcolor=blue,citecolor=blue,hyperfootnotes=false,pagebackref]{hyperref}
\newcommand{\R}{\mathbb R}
\newcommand{\dd}{\,d}
\newcommand{\supp}{\operatorname{supp}}
\renewcommand{\Cap}{\operatorname{Cap}}

\newtheorem{theorem}{Theorem}[section]
\newtheorem{lemma}[theorem]{Lemma}

\theoremstyle{definition}

\theoremstyle{remark}

\allowdisplaybreaks

\numberwithin{equation}{section}

\begin{document}
	\begin{center}
		
		{\Large \bf Sharp Riesz capacity inequalities and the P\'olya--Szeg\H{o} conjecture}
		\vskip 15pt
		\begin{center}
			{{\bf    Mengwei~~Su$^1$ \quad Qiang~~Sun$^2$ \quad Ge Xiong$^1$}\\~~ \\ \small{$^1$ School of Mathematical Sciences, Key Laboratory of Intelligent Computing and Applications (Tongji University), Ministry of Education, Tongji University, Shanghai, 200092, PR China\\
					$^2$ Shanghai Center for Mathematical Sciences, Fudan University, Shanghai, 200438, PR China}}
		\end{center}
	\end{center}
	
	\vskip 3pt

	\footnotetext{Emails: smw@tongji.edu.cn; sunqiang@fudan.edu.cn; xiongge@tongji.edu.cn.}
	\footnotetext{Research of the second author is supported by China National Postdoctoral Program for Innovative Talents No. BX20250071.}
	
	\begin{center}
		\begin{minipage}{15cm}
			{\bf Abstract} 
			For $n\geq2$ and $p\in(n-2,n)$, we prove that closed balls maximize
			the ratio of Riesz $p$-capacity to Newtonian capacity among all compact
			subsets of $\mathbb{R}^n$.
			This resolves the classical P\'olya--Szeg\H{o} capacity ratio conjecture
			and its higher-dimensional extensions.
			For $n\geq4$,  a sharp inequality for the ratio of Newtonian
			capacity to  logarithmic capacity is established.
			All equality cases are characterized.

			
			
		\end{minipage}
	\end{center}
	
	
	\vskip 25pt
	\section{\bf Introduction}
	\vskip 5pt
	
	
	In 1945, P\'olya and Szeg\H{o}
	\cite[Conjecture~(1.3)]{polya-szego} conjectured that for every nonempty
compact set $K\subseteq\R^2$,
	\begin{equation}\label{PSI}
		\Cap_1(K)\leq\frac{2}{\pi}\Cap_0(K),
	\end{equation}
where $\Cap_1$ is the
Newtonian capacity (i.e., electrostatic capacity) when the planar set $K$ is viewed as a subset
of $\R^3$, and $\Cap_0$ is the logarithmic capacity.
	The constant $\frac{2}{\pi}$ is sharp, since a closed disk $B$ of radius $r>0$
	satisfies $\Cap_0(B)=r$ and $\Cap_1(B)=\frac{2r}{\pi}$.
	
	Earlier, in 1931, P\'olya and Szeg\H{o} \cite[p.~19]{polya-szego-1931}
	showed that $\Cap_1(K)\leq\Cap_0(K)$ for  compact
	sets in $\mathbb{R}^2$.
    Nearly two decades later, using conformal mapping and Dirichlet energy, Szeg\H{o} \cite{szego-1952,szego-1956}  proved that every
    connected compact set in $\mathbb{R}^2$ satisfies
    \[
    \Cap_1(K)\leq1.07\,\frac{2}{\pi}\Cap_0(K).
	\]
	
  Let $n\ge 2$ and $K$ be a nonempty compact set in $\mathbb{R}^n$. Write $\mathcal P(K)$ for the set of Borel probability measures supported on $K$. 
	For $p\in(0,n)$, the \emph{Riesz $p$-capacity} of $K$ is
	\begin{align*}
		\Cap_p(K)
		=
		\big(
		\inf_{\mu\in\mathcal P(K)}
		\int_{K}\int_{K}{|x-y|^{-p}} \dd\mu(x)\dd\mu(y)
		\big)^{-\frac{1}{p}},
	\end{align*}
	and the \emph{logarithmic capacity} of $K$ is 
	\begin{align*}
		\Cap_0(K)
		=
		\exp\big(
		-\inf_{\mu\in\mathcal P(K)}
		\int_K\int_K
		\log\frac{1}{|x-y|}\,\dd\mu(x)\dd\mu(y)
		\big).
	\end{align*}

    Let $K\subseteq\R^n$ be a Borel set  and $p\in[0,n)$. It is said that
	$\Cap_p(K)=0$, if
	$\Cap_p(L)=0$ for every nonempty compact set $L\subseteq K$. For more details on Riesz $p$-capacity, refer to the excellent books \cite{landkof} by Landkof and \cite{bhs} by Borodachov, Hardin and Saff.

		Let $n\geq3$. 
The power $n-2$ of the Riesz $(n-2)$-capacity, up to a constant, is precisely the classical
\emph{Newtonian capacity} (also called \emph{electrostatic  capacity}), which is expressed as 
	\[
	n(n-2)\omega_n\Cap_{n-2}(K)^{n-2}=
	\inf\{
	\int_{\R^n}|\nabla\varphi|^2\,dx:
	\varphi\in C_c^\infty(\R^n),\
	\varphi\geq1\text{ on }K
	\},
	\]
	where $\omega_n=\frac{\pi^\frac{n}{2}}{\Gamma (\frac{n}{2}+1)}$  is the volume of the standard Euclidean
	unit ball in $\mathbb{R}^n$.
Newtonian capacity gives the least energy needed to hold \(K\) at unit potential, with zero-capacity sets negligible in Sobolev and elliptic potential theory. 

It is interesting that 
	for convex bodies, the variational structure of electrostatic  capacity  has given rise to a fruitful Brunn--Minkowski theory: In 1983,
	Borell \cite{borell} established the Brunn--Minkowski inequality for electrostatic capacity; In 1996,
	Jerison \cite{jerison} proved the existence and regularity of solutions to the Minkowski problem for  electrostatic capacity; In the same year, Caffarelli, Jerison and Lieb \cite{caffarelli-jerison-lieb} characterized the equality cases in Borell's inequality, and therefore obtained the uniqueness;
	In 2020, the third author of this article and Zou \cite{zou-xiong-jdg} proved the existence and uniqueness of solutions to the $L_p$ Minkowski problem for electrostatic $\mathfrak p$-capacity with $p>1$ and $1<\mathfrak p<n$.
	Please refer to,  e.g., \cite{akman-et-al,akman-lewis-saari-vogel,colesanti-variational,colesanti-cuoghi,colesanti-et-al,colesanti-salani,lu-xiong,xiong-xiong,xiong-xiong-xu,zou-xiong-cag} for
	 further developments.

	\vskip3pt
	
For compact sets with prescribed volume, Watanabe \cite{watanabe}
proved that balls minimize Riesz $p$-capacity $\Cap_p$ for $p\in(n-2,n)$.
See  Betsakos \cite{betsakos,betsakos-addendum} and
M\'endez-Hern\'andez \cite{mendez-hernandez} for different proofs.
For planar polygons, Solynin and Zalgaller
\cite{solynin-zalgaller} proved that the regular polygon minimizes
logarithmic capacity if  the area and number of sides
are fixed.

Very recently,  Clark and Laugesen \cite[Conjecture~1.1]{toward} posed the following conjecture, which naturally extends
\eqref{PSI} to all dimensions:
Let $n\geq2$. Among all compact $K\subseteq\R^n$ with
$\Cap_{n-2}(K)>0$, the closed $n$-dimensional balls maximize the ratio
\begin{equation}\label{eq:newtonian-ratio-conjecture}
	\frac{\Cap_{n-1}(K)}{\Cap_{n-2}(K)}.
\end{equation}
They showed that this conjecture  follows from the transverse energy conjecture
\cite[Conjecture~1.2 and Theorem~1.3]{toward}.

	\vskip3pt

	In this paper, we resolve the P\'olya--Szeg\H{o} capacity ratio conjecture \eqref{PSI} and the Clark-Laugesen conjecture
	\eqref{eq:newtonian-ratio-conjecture} for all $n\geq2$.
	More generally,  for every
	$p\in(n-2,n)$ and  compact $K\subseteq\R^n$ with
	$\Cap_{n-2}(K)>0$, we obtain the sharp upper bound
	of $$\frac{\Cap_p(K)}{\Cap_{n-2}(K)}$$ 
	and characterize the equality condition (Theorem \ref{thm:main}).
	We also obtain a sharp inequality for the ratio of Newtonian
	capacity to  logarithmic capacity when $n\geq4$ (Theorem~\ref{thm:newton-log}).

	For brevity, we fix the following notations. Let $n\geq2$ and $p\in(n-2,n)$. Write
	\begin{align*}
		c_{n,p}=\big(\frac{\Gamma(\frac{n}{2})}{\Gamma(1+\frac{p}{2})\Gamma(\frac{n-p}{2})}\big)^{\frac{1}{p}}\quad \text{and} \quad c_n=\frac{1}{2}\exp\!\big(\frac{\Gamma'(n-1)}{2\Gamma(n-1)}-\frac{\Gamma'(\frac{n-1}{2})}{2\Gamma(\frac{n-1}{2})}\big),
	\end{align*}
	where $\Gamma(x)=\int_0^\infty t^{x-1}e^{-t}\,dt$, and $
	\Gamma'(x)=\int_0^\infty t^{x-1}e^{-t}\log t\,dt.$

	\begin{theorem}\label{thm:main}
		If $n\geq2$, $p\in(n-2,n)$ and $K$ is a nonempty compact set in $\mathbb{R}^n$, then 
		\begin{equation*}\label{eq:main}
			\Cap_p(K)
			\leq
			c_{n,p}\Cap_{n-2}(K)
		\end{equation*}
		with equality if and only if $\Cap_{n-2}(K\setminus B)=0$ for some closed ball $B\subseteq K$.
	\end{theorem}
	
	\begin{theorem}\label{thm:newton-log}
		If $n\geq4$ and $K$ is a nonempty compact set in $\mathbb{R}^n$, then
		\begin{equation*}\label{eq:newton-log}
			\Cap_{n-2}(K)\leq c_n\Cap_0(K)
		\end{equation*}
		with equality if and only if $\Cap_0(K\setminus B)=0$ and $\partial B\subseteq K$ for some closed ball $B$.
	\end{theorem}

If $n=2$ and $p=1$  in Theorem \ref{thm:main}, then $c_{2,1}=\frac{2}{\pi}$. Therefore, we resolve the classical P\'olya--Szeg\H{o} capacity ratio inequality \eqref{PSI}.
If $n\ge 2$ and $p=n-1$, we give an affirmative answer to  the Clark-Laugesen conjecture
\eqref{eq:newtonian-ratio-conjecture}. Since $\lim\limits_{p\to n^-}\Cap_p(K)/c_{n,p}=(|K|/\omega_n)^\frac{1}{n}$ (see \cite[Corollary~1.2]{continuity}),
letting $p\to n^-$ in Theorem~\ref{thm:main}, we obtain
the classical isocapacitary inequality $|K|\leq \omega_n\Cap_{n-2}(K)^n$,
where $|K|$ is the $n$-dimensional volume of $K$.
	
	We sketch the idea underlying the proof of Theorem~\ref{thm:main}.
	First, we establish a sharp \emph{moment inequality} for a Borel measure $\sigma$ with its Riesz $p$-potential
	$U_p^\sigma(x)\leq |x|^{-p}$, and characterizes the equality condition (Lemma~\ref{lem:moment}).
Second, we define an auxiliary function $F$ using the inverse
normalized incomplete beta function.
It is remarkable that with this $F$,  we shows that
$F\circ U_p^\sigma$ is \emph{superharmonic} outside $\supp\sigma$ by the established moment
inequality (Lemma~\ref{lem:gradient}).
Finally, taking $\sigma$ to be the $p$-capacitary measure of $K$,
we construct a new function $W$ by adding $F\circ U_p^\sigma$ and the Riesz $(n-2)$-potential of $K$, in order that  $W$ is still \emph{superharmonic} on the unbounded component
of $\R^n\setminus K$ and 
\[|x|^{n-2}W(x)=	\begin{cases}
	\log\big(
\frac{c_{2,p}\Cap_0(K)}{\Cap_p(K)}
\big)
+
O(|x|^{-1}), & n=2,\\[6pt]
	\frac{
		\Cap_{n-2}(K)^{n-2}
		-
		c_{n,p}^{-(n-2)}\Cap_p(K)^{n-2}}{n-2}
	+
	O(|x|^{-1}),
	& n\geq3,
\end{cases}\quad \text{as } |x|\to \infty.
\]
By the \emph{generalized minimum principle} for superharmonic functions (Lemma \ref{lem:exterior-minimum}), we have $W\geq0$, which yields the desired inequality.

The proof of Theorem~\ref{thm:newton-log} depends on space dimensions.
For $n\geq5$, using the Riesz composition
formula (\ref{lem:riesz-composition}), the Riesz rearrangement inequality, and the rearrangement
inequality for Riesz $(n-2)$-potentials  (Lemma~\ref{lem:distribution}), we obtain the desired inequality.
For $n=4$,
we derive an asymptotic expansion
for $\int_{B_o(R)}U^2\,\dd x$ and prove that
   $\int_{B_o(R)}(U^*)^2\,\dd x-\int_{B_o(R)}U^2\,\dd x=O(R^{-1})$, where $U$ is the  Riesz $(n-2)$-potential of $K$ and $U^*$ is the symmetric decreasing rearrangement of $U$.
By cancelling the divergent terms and letting $R\to\infty$, we establish the desired inequality.

	This paper is organized as follows. In Section~\ref{sec:preliminaries}, we collect some basic facts on potential theory and prove several useful lemmas. Theorem~\ref{thm:main} and Theorem~\ref{thm:newton-log} are proved in Section~\ref{sec:moment-method} and Section~\ref{sec:first-proof}, respectively.
	
		\vskip15pt
	\section{\bf Preliminaries}\label{sec:preliminaries}
		\vskip5pt
	Write $o$ for the origin of $\mathbb{R}^n$ and $x\cdot y$ for the standard inner product of $x,y\in\mathbb{R}^n$.  Let $\mathbb{S}^{n-1}=\{x\in\mathbb{R}^n: |x|=1\}$ be the boundary of the Euclidean unit ball in $\mathbb{R}^n$. 
	 For $x\in\mathbb{R}^n$ and $r\geq0$, write $B_x(r)$ for the closed ball with center $x$ and
	radius $r$.

	Let $n\ge 2$ and $p\in[0,n)$. The \emph{Riesz $p$-kernel} is
	\[
	k_p(x)=
	\begin{cases}
		|x|^{-p}, & p>0,\\
		-\log|x|, & p=0,
	\end{cases}
	\]
	and the \emph{Riesz $p$-potential} of a finite Borel measure $\mu$ is $$ U_{p}^{\mu}(x)=\int_{\R^n} k_p(x-y)\,\dd\mu(y).$$
	For a nonempty compact set $K\subseteq\R^n$ and $\mu\in\mathcal P(K)$, define
	\[
	I_p(\mu)
	=
	\int_{\R^n}U_{p}^{\mu}(x) \dd\mu(x),\]
	and the \emph{Riesz $p$-energy} of $K$ is $
	V_p(K)
	=
	\inf_{\mu\in\mathcal P(K)}I_p(\mu).$
	A measure $\mu_p\in\mathcal P(K)$ attaining this infimum is called
	the \emph{$p$-equilibrium measure} of $K$. 
	Thus, the Riesz $p$-capacity of $K$ is
	\[
	\Cap_p(K)=
	\begin{cases}
		V_p(K)^{-\frac{1}{p}}, & p>0,\\
		e^{-V_0(K)}, & p=0.
	\end{cases}
	\]
	
	It is known that $\Cap_{p}(tK+z) = t\Cap_{p}(K)$ for $t>0$ and $z\in \mathbb{R}^n$.

	A property holds \emph{$p$-quasi-everywhere} if it holds outside
	some Borel set of zero $p$-capacity.
	
	Let $n\geq2$ and $p\in(n-2,n)$.  For any closed ball $B$ with positive
	radius, we have the following formula (See, e.g., \cite[Appendix~A]{continuity}) 
	\begin{equation}\label{eq:ball-capacity-ratio}
		c_{n,p}
		=\big(\frac{\Gamma(\frac{n}{2})}{\Gamma(1+\frac{p}{2})\Gamma(\frac{n-p}{2})}\big)^{\frac{1}{p}}=\frac{\Cap_p(B)}{\Cap_{n-2}(B)}.
	\end{equation}
Write $\beta$ for the {\emph{normalized incomplete beta function}}
\[
\beta(t)
=
\frac{
	\int_0^t
	u^{\frac{p}{2}-1}(1-u)^{\frac{n-p}{2}-1}\,\dd u
}{\int_0^1
	u^{\frac{p}{2}-1}(1-u)^{\frac{n-p}{2}-1}\,\dd u}, \quad t\in [0,1].
\]
Then $\beta$ is a strictly increasing bijection on $[0,1]$.
We denote $\phi$ as the inverse of $\beta$.
	
		A Borel set $K\subseteq\R^n$ is called \emph{polar} if there exist an open set
	$L\supseteq K$ and a superharmonic function $f$ on $L$
	such that $f=+\infty$ on $K$, and $f$ is not identically $+\infty$
	on any connected component of $L$.
	See \cite[Definition~4.2.1]{helms}.
	For $n\geq3$ and a Borel set $K\subseteq\R^n$,
	\begin{equation}\label{eq2.6}
		K\text{ is polar}
		\quad\Longleftrightarrow\quad
		\Cap_{n-2}(K)=0.
	\end{equation}
	Refer to \cite[Theorems~4.4.26 and~4.5.2]{helms}.
	
	The following lemma collects several standard facts from potential theory.
	Please refer to \cite[Chapter II]{landkof}, \cite[Section 4]{bhs} and \cite[Appendix A.3]{csw}. 
	
	\begin{lemma}\label{lem:variation-polar}
		If $K$ is a nonempty compact set in $\mathbb{R}^n$
		with $V_p(K)<\infty$ and  $p\in[0,n)$, then $K$ admits a unique $p$-equilibrium
		measure $\mu_p\in\mathcal P(K)$. Moreover, $U_{p}^{\mu_p}\geq V_p(K)$ $p$-quasi-everywhere on $K$ and $U_{p}^{\mu_p}=V_p(K)$ $\mu_p$-almost everywhere.
	\end{lemma}
	
	Under the assumptions of Lemma~\ref{lem:variation-polar},
	if in addition $p\in(0,n)\cap[n-2,n)$, by the Frostman maximum
	principle \cite[Theorem~2.2]{zorii}, we have
	\begin{align}\label{Lem 2.1 cor}
		0<U_p^{\mu_p}\leq V_p(K) \text{ on } \mathbb{R}^n, \quad \text{ and } \quad U_{p}^{\mu_p}=V_p(K)
		\text{ $p$-quasi-everywhere on }K.
	\end{align}
	
	Let $\eta_p
	=
	\frac{\mu_p}{V_p(K)}
	=
	\Cap_p(K)^p\mu_p$ be  the \emph{$p$-capacitary measure} of $K$.
    Then, 
	\begin{align}\label{eq:equilibrium-voltage}
		0<U_{p}^{\eta_p}\leq1 \text{ on } \mathbb{R}^n, \quad \text{ and } \quad U_{p}^{\eta_p}=1 \text{ $p$-quasi-everywhere on } K.
	\end{align}
	
	The following lemma follows from
	\cite[Theorem~1.1, Remark~2.5, and Eq.~(A.10)]{csw}
	by rescaling the equilibrium measure so that its potential
	equals $1$ on the ball $B_o(s)$.
	\begin{lemma}\label{lem:ball}
		If $n\geq2$, $p\in(n-2,n)$, $s>0$ and $\eta_p$
		is the $p$-capacitary measure of $B_o(s)$, then $\eta_p(\R^n)=c_{n,p}^p s^p$, $	\dd\eta_p(y)
		=\pi^{-\frac{n}{2}-1}\sin(\frac{n-p}{2}\pi)	\Gamma(\frac{n}{2})
		(s^2-|y|^2)^{-\frac{n-p}{2}}
		\mathbf 1_{\{|y|<s\}}\,\dd y$
		and 
		\begin{equation*}
			U_{p}^{\eta_p}(x)
			=
			\begin{cases}
				1, & 0\leq |x|\leq s,\\
				\beta(\frac{s^2}{|x|^2}), & |x|>s.
			\end{cases}
		\end{equation*}
	\end{lemma}

	\begin{lemma}\label{lem:capacity-polar-removal}
		Let $n\geq2$, $p\in[0,n)$ and $K$ be a nonempty
		compact set in $\mathbb{R}^n$.
		If $L\subseteq K$ is a nonempty compact set with
		$\Cap_p(K\setminus  L)=0$, then $\Cap_p(K)=\Cap_p(L)$.
	\end{lemma}
	
	\begin{proof}
		Since $\Cap_p(K\setminus L)=0$, it follows that  every nonempty compact
		subset of $K\setminus L$ has $p$-capacity zero. Applying
		\cite[p.~158, Corollary]{fuglede} with $X=K$, $A=L$ and $N=K\setminus L$, we have
	 $V_p(K)=V_p(L)$, and hence $\Cap_p(K)=\Cap_p(L)$.
	\end{proof}

	\begin{lemma}
		\label{lem:exterior-minimum}
		Let $n\geq2$, $K$ be a nonempty compact set in $\mathbb{R}^n$ and
		$\Omega_K$ be the unbounded component of $\R^n\setminus K$.
		Assume  $W$ is superharmonic and bounded below on $\Omega_K$,
		and
		\begin{equation*}
			\label{eq-lem2.4}\liminf_{x\in\Omega_K,\, x\to x_0}W(x)\geq0,
			\quad\text{for $(n-2)$-quasi-every}\ 
			x_0\in\partial\Omega_K.
		\end{equation*}

		\noindent (1) If $n=2$ and $\Cap_0(\partial\Omega_K)>0$, then
		$W\geq 0$ on $\Omega_K$. 
		
		\noindent (2) If $n\geq3$ and $\liminf\limits_{|x|\to\infty}W(x)\geq0$, then $W\geq 0$ on $\Omega_K$.
	\end{lemma}
	
	\begin{proof}
		If $n=2$, by \cite[Theorem 2.6]{saff}, it follows that $W\geq 0$ on $\Omega_K$.
		
		For $n\geq3$, by 	\cite[Theorem~3.2.10]{helms}, it follows that every open subset of $\R^n$ is Greenian.
		In particular, $\Omega_K$ is Greenian.
		Combining \eqref{eq2.6} with
		the generalized minimum principle
		\cite[Theorem~4.2.21]{helms}, it follows that $W\geq0$ on $\Omega_K$.
	\end{proof}
	\begin{lemma}\label{lem:riesz-log-limit}
		Let $\mu$ be a compactly supported probability measure on $\R^n$.
		If $I_s(\mu)<\infty$ for some $s\in(0,n)$, then
		 $I_p(\mu)<\infty$ for all $p\in[0,s]$ and $	\lim\limits_{p\to0^+}\frac{I_p(\mu)-1}{p}=I_0(\mu).$
	\end{lemma}
	
	\begin{proof}
		Since $|z|^{-p} \leq 1 + |z|^{-s}$ for all $ z\in\mathbb{R}^n\setminus\{o\}$ and $I_{s}(\mu) <\infty$, we have for all $ p\in(0,s]$,
		\begin{align*}
			& I_p(\mu) = \int_{\mathbb{R}^n}\int_{\mathbb{R}^n}  |x-y|^{-p}d\mu(x)d\mu(y) \\
			 \leq & \int_{\mathbb{R}^n}\int_{\mathbb{R}^n} (1+|x-y|^{-s}) d\mu(x)d\mu(y) = 1+I_s(\mu)<\infty.
		\end{align*}
		
		Write $T=\max\{1,\operatorname{diam}(\supp\mu)\}$ and 
		\[
		f_p(x,y)=
		\begin{cases}
			\displaystyle\frac{|x-y|^{-p}-1}{p}, & x\neq y,\\[6pt]
			0, & x=y,
		\end{cases}
		\qquad \forall x,y\in\mathbb{R}^n.
		\]
		Then, $
		\lim_{p\to0^+}f_p(x,y)
		=-\log|x-y|$ for $x\neq y.$
		
		If $x,y \in \operatorname{\supp}\mu$ and $|x-y|\leq 1$, then $|f_{p}(x,y)| \leq \frac{|x-y|^{-s}}{s}$; If $x,y \in \operatorname{\supp}\mu$ and $|x-y|> 1$, then $|f_{p}(x,y)| \leq \frac{1-e^{-p\log T}}{p} \leq \log T$.
		So,
		\begin{align*}
			&|\frac{I_p(\mu)-1}{p}| \leq\int_{\mathbb{R}^n}\int_{\mathbb{R}^n}	|f_{p}(x,y)|\, d\mu(x)d\mu(y) \\
			 \leq &\int_{\mathbb{R}^n}\int_{\mathbb{R}^n} (\frac{|x-y|^{-s}}{s} + \log T)\, d\mu(x)d\mu(y) \leq \frac{I_{s}(\mu)}{s}+\log T < \infty.
		\end{align*}
		
		By the dominated convergence theorem, we have
		\begin{align*}
			&\lim_{p\to0^+}\frac{I_p(\mu)-1}{p}
			=
			\lim_{p\to0^+}
			\int_{\mathbb{R}^n}\int_{\mathbb{R}^n}	f_{p}(x,y)\, d\mu(x)d\mu(y) \\
			=&
			\int_{\mathbb{R}^n}\int_{\mathbb{R}^n}
			-\log|x-y|
			\, d\mu(x)d\mu(y)
			=I_0(\mu).
		\end{align*}
		Moreover, $I_0(\mu) \leq \frac{I_{s}(\mu)}{s}+\log T < \infty.$
	\end{proof}

	\section{\bf Proof of Theorem~\ref{thm:main}}\label{sec:moment-method}
	\vskip5pt
	Throughout this section, let $n\geq2$ and $p\in(n-2,n)$. 
	
	For brevity, write $g(x) = c_{n,p}^p \phi(x)^{\frac{p}{2}}
	\bigl(1-\phi(x)\bigr)^{\frac{n-p}{2}}$, $x\in[0,1]$, where $\phi$ is the inverse of the function $\beta$  introduced in
	Section~\ref{sec:preliminaries}. Then $g(x)=0$ if and only if $x=0$ or $x=1$.
	
	Write $\delta_o$ for the Dirac measure at the origin $o$.

	\begin{lemma}\label{lem:moment}
		Let $\sigma$ be a Borel measure on $\R^n$ with $m=\sigma(\mathbb{R}^n)<1$ and $\int_{\R^n} |x|^2\,\dd\sigma(x) < \infty$. 
		If $U_{p}^{\sigma}(x)\leq |x|^{-p}$ for $x\neq o$, then
		\begin{equation}\label{eq:moment}
			\big|\int_{\mathbb{R}^n} x\,\dd\sigma(x)\big|^2
			\leq
			g(m) \int_{\mathbb{R}^n} |x|^2\,\dd\sigma(x)
		\end{equation}
		with equality if and only if either $\sigma=m\delta_o$, or $v:=\int_{\R^n} x\,\dd\sigma(x)\neq o$, $\supp\sigma\subseteq B_{z}(r)$ and $U_p^\sigma(x)=|x|^{-p}$ for Lebesgue-a.e. $x\in B_{z}(r)$,
		where $z=g(m)^{-1}v$ and $r=g(m)^{-1}\phi (m)^\frac{1}{2}|v|$. 
	\end{lemma}
	
	\begin{proof}
		\smallskip
		
		If $m=0$, then $\sigma=0$ and the equality holds. If $0<m<1$ and
		$v=o$, then \eqref{eq:moment} holds. 
		Since $g(m)>0$, the equality
		holds precisely when $\int_{\R^n} |x|^2\,\dd\sigma(x)=0$, i.e.,
		$\sigma=m\delta_o$.
		
		In the following,  assume that $0<m<1$ and $v\neq o$. Then $g(m)>0$.
	Let $\lambda$ be  the Borel measure such that for every nonnegative Borel measurable function $f$ on $\R^n$,
	\[
	\int_{\R^n} f(y)\,\dd\lambda(y)
	=
	\frac{2n}{p}\int_0^r
	s\big(\int_{\R^n} f(y)\,\dd\eta_p(y)\big)\,\dd s,
	\]
	where $\eta_p$ denotes the $p$-capacitary
	measure of $B_o(s)$.
		Since $U_{p}^{\lambda}(x)
		=\int_{\R^n}|x-y|^{-p}\,\dd\lambda(y)$,
		we have
		$
		U_p^\lambda(o)
		=\frac{2n}{p}\int_0^r s\,\dd s
		=\frac{n}{p}r^2$.
		
		For $0<|x|\leq r$, by Lemma \ref{lem:ball} we have
		\begin{align*}
			&U_{p}^{\lambda}(x)
			=
			\frac{2n}{p}\int_0^r
			s\,U_{p}^{\eta_p}(x)\,\dd s 
			=
			\frac{2n}{p}\int_0^{|x|}
			s\beta(\frac{s^2}{|x|^2})\dd s
			+
			\frac{2n}{p}\int_{|x|}^r s\,\dd s \\
			=&
			\frac{n}{p}|x|^2\int_0^1 \beta(t)\,\dd t
			+
			\frac{n}{p}\bigl(r^2-|x|^2\bigr) 
			= \frac{n}{p}|x|^2 \frac{n-p}{n} + \frac{n}{p}\bigl(r^2-|x|^2\bigr) 
			= \frac{n}{p} r^2-|x|^2.
		\end{align*}
		For $|x|>r$, by Lemma \ref{lem:ball} and $\beta (t)< 1$ for all $t\in(0,1)$,  we have
		\begin{align*}
			\ U_p^\lambda(x) =& \frac{2n}{p}\int_0^r
			s\,U_{p}^{\eta_p}(x)\,\dd s 
			=  \frac{2n}{p}\int_0^{|x|} 
			s \beta (\frac{s^2}{|x|^2})\,\dd s - \frac{2n}{p}\int_r^{|x|}
			s \beta (\frac{s^2}{|x|^2})\,\dd s \\
			>&\frac{n-p}{p}|x|^2 - \frac{2n}{p} \int_{r}^{|x|} s \,\dd s = \frac{n}{p} r^2-|x|^2.
		\end{align*}
		Consequently, we have $U_{p}^{\lambda}(x)
		\geq
		\frac{n}{p} r^2-|x|^2$, for all $x\in\R^n$.
		So,
		\begin{align}\label{eq3.2}
		&\ \ \ \ \int_{\R^n}
		U_{p}^{\lambda}(z-x)\,\dd\sigma(x)
		\geq
		\int_{\R^n}
		\big(
		\frac{n}{p} r^2-|x-z|^2
		\big)\dd\sigma(x)\\
	\nonumber	&=
		\frac{n}{p} |z|^2\phi(m)m
		-\int_{\R^n}|x|^2\,\dd\sigma(x)
		+2|z||v|-|z|^2m
		\end{align}
	with equality if and only if $\supp \sigma \subseteq B_z(r)$. 
		
		Meanwhile, from $U_{p}^{\sigma}(x)\leq |x|^{-p}$ for $x\neq o$, Lemma \ref{lem:ball} and $|z|>r$, we have
		\begin{align}
			\label{eq3.3}&\ \int_{\R^n}
			U_{p}^{\lambda}(z-x)\,\dd\sigma(x)
			=\ 
			\int_{\R^n}
			U_{p}^{\sigma}(z-y)\,\dd\lambda(y)
			\leq
			\int_{\R^n}
			|z-y|^{-p}\,\dd\lambda(y)\\
			\nonumber=&\ 
			U_{p}^{\lambda}(z)
			=
			\frac{2n}{p}\int_0^r
			s\beta(\frac{s^2}{|z|^2})\dd s 
			= 
			\frac{n}{p} |z|^2
			\int_0^{\phi(m)}\beta(u)\,\dd u\\
			\nonumber	=&\
			\frac{n}{p} |z|^2
			\big(
			\phi(m)m
			-
			\frac{
				\int_0^{\phi(m)}
				u^{\frac{p}{2}}(1-u)^{\frac{n-p}{2}-1}\,\dd u}
			{\int_0^1
				u^{\frac{p}{2}-1}(1-u)^{\frac{n-p}{2}-1}\,\dd u}
			\big).
		\end{align}
		
		Combining \eqref{eq3.2} and \eqref{eq3.3}, together with that $$
		g(m)=m-
		\frac{n
			\int_0^{\phi(m)}
			u^{\frac{p}{2}}(1-u)^{\frac{n-p}{2}-1}\,\dd u}
		{p\int_0^1
			u^{\frac{p}{2}-1}(1-u)^{\frac{n-p}{2}-1}\,\dd u},
		$$ it follows that $	\int_{\R^n}|x|^2\,\dd\sigma(x)
		\geq
		2|z||v|
		-
		|z|^2g(m).$
		That is,
		\begin{align*}
			\big|\int_{\R^n} x\,\dd\sigma(x)\big|^2
			\leq
			g(m) \int_{\R^n} |x|^2\,\dd\sigma(x).
		\end{align*}
		
		It remains to characterize the  equality conditions. The equality in \eqref{eq:moment} holds iff the equalities in \eqref{eq3.2} and \eqref{eq3.3}  hold simultaneously, and iff $\supp\sigma\subseteq B_{z}(r)$ and $U_p^\sigma(z-y)
		=|z-y|^{-p}$ for $\lambda$-almost every $y\in\mathbb{R}^n$.
		In addition, $\supp\lambda\subseteq B_o(r)$ and $\lambda$ is absolutely
		continuous and has a strictly positive density on $\operatorname{int}B_o(r)$  by Lemma \ref{lem:ball}. Thus, the equality in \eqref{eq:moment} holds iff $\supp\sigma\subseteq B_{z}(r)$ and $U_p^\sigma(z-y)
		=|z-y|^{-p}$ for Lebesgue-almost every $y\in B_o(r).$
		This completes the proof.
	\end{proof}

	Define $F:(0,1)\to\R$ by
	\begin{align}
	\label{defofF}	F(t)=
		\begin{cases}
			-\frac{1}{2}\log \phi(t),&n=2,\\[5pt]
			\frac{1-\phi(t)^{\frac{n-2}{2}}}{n-2},&n\geq3.
		\end{cases}
	\end{align}
	
	\begin{lemma}\label{lem:gradient}
		If $\sigma$ is a nonzero finite Borel measure on $\R^n$, $x\in\R^n\setminus\supp\sigma$ and $U:= U_{p}^{\sigma} \leq 1$ on $\R^n$,
		then  $U(x)\in (0,1)$,  $\Delta (F\circ U)(x)\leq 0$ and
		$\Delta (F\circ U)(x)= 0$ if and only if  $\sigma$ is the $p$-capacitary measure of  some closed ball $B$ with positive radius and $x\notin B$.
	\end{lemma}

		\begin{proof}
			
			\textbf{Step 1.} Prove that $U(x)\in(0,1)$ and $\Delta(F\circ U)(x)\leq0$.
			
			Since $\sigma(\R^n)>0$, it follows that for every
			$x\in\R^n\setminus\supp\sigma$,
			\begin{align*}
				U(x)=\int_{\R^n}|y-x|^{-p}\,\dd\sigma(y)>0,\quad
				\Delta U(x)=p(p-n+2)\int_{\R^n}|y-x|^{-p-2}\,\dd\sigma(y)>0.
			\end{align*}
			Assume $U(x)=1$. Added with that $U\leq1$ on $\R^n$, it follows that
			$x$ is a maximum of $U$ and hence $\Delta U(x)\leq0$,
			contradicting $\Delta U(x)>0$. Therefore, $U(x)\in(0,1)$.
			
			Let $\lambda$ be the Borel measure such that for every nonnegative Borel measurable function $f$ on $\R^n$,
			\[
			\int_{\R^n} f(y)\,\dd\lambda(y)
			=
			\int_{\R^n}
			f\big(\frac{y-x}{|y-x|^2}\big)
			|y-x|^{-p}\,\dd\sigma(y).
			\]

			Three observations are in order. First, by $U\leq1$ on $\R^n$,
			we have for $\xi\neq o$,
			\begin{align*}
				U_p^\lambda(\xi)=
				\int_{\R^n}
				\big|\xi-\frac{y-x}{|y-x|^2}\big|^{-p}
				|y-x|^{-p}\,\dd\sigma(y)=
				|\xi|^{-p}
				U\big(x+\frac{\xi}{|\xi|^2}\big)
				\leq|\xi|^{-p}.
			\end{align*}
			Second, by $x\notin\supp\sigma$ and $\sigma(\mathbb{R}^n)<\infty$,
			we have $\int_{\R^n}|y|^2\dd\lambda(y)
			\!=\!
			\int_{\R^n}|y-x|^{-p-2}\dd\sigma(y)\!<\!\infty.$
			Third, $	\lambda(\R^n)
			=
			\int_{\R^n}|y-x|^{-p}\,\dd\sigma(y)
			=
			U(x)\in(0,1).$ Thus,
			by Lemma~\ref{lem:moment} we obtain
			\begin{align}\label{eq3.4}
				\big|\int_{\R^n}y\,\dd\lambda(y)\big|^2
				\leq
				g\big(U(x)\big)
				\int_{\R^n}|y|^2\,\dd\lambda(y).
			\end{align}
			Combining this with the following identities
			\begin{align*}
				\nabla U(x)
				&=
				p\int_{\R^n}(y-x)|y-x|^{-p-2}\,\dd\sigma(y)
				=
				p\int_{\R^n}y\,\dd\lambda(y),\\
				\Delta U(x)
				&=
				p(p-n+2)\int_{\R^n}|y-x|^{-p-2}\,\dd\sigma(y)
				=
				p(p-n+2)\int_{\R^n}|y|^2\,\dd\lambda(y),
			\end{align*}
			we obtain that
			\begin{align}\label{eq:gradient}
				\frac{|\nabla U(x)|^2}{p^2}
				\leq
				g(U(x))
				\frac{\Delta U(x)}{p(p-n+2)}.
			\end{align}
			
			By the definition of $F$ and the identity
			$\beta(\phi(t))=t$, we have
			\[
			F'(t)
			=
			-\frac{1}{2}
			\big(
			\int_0^1
			u^{\frac{p}{2}-1}(1-u)^{\frac{n-p}{2}-1}\,\dd u
			\big)
			\phi(t)^{\frac{n-p}{2}-1}
			\big(1-\phi(t)\big)^{\frac{p-n+2}{2}}
			<0
			\]
			and
			\[
			\frac{F'(t)}{F''(t)}
			=
			-\frac{p}{p-n+2}\,g(t),
			\qquad 0<t<1.
			\]
			Since $g(t)>0$, it follows that $F''(t)>0$.
			So, for $x\in\R^n\setminus\supp\sigma$,
			by \eqref{eq:gradient} we have
			\begin{align*}
				\Delta(F\circ U)(x)
				&=
				F''(U(x))|\nabla U(x)|^2
				+
				F'(U(x))\Delta U(x)\\
				&\leq
				\big(
				F''(U(x))
				\frac{p}{p-n+2}\,g(U(x))
				+
				F'(U(x))
				\big)\Delta U(x)
				=0.
			\end{align*}
			That is, $\Delta(F\circ U)(x)\leq0$.
			
			\textbf{Step 2.} Characterize the equality condition.
			
			Suppose that $\sigma$ is the $p$-capacitary measure of
			$B=B_w(s)$, for some $w\in\R^n$ and $s>0$.
			By Lemma~\ref{lem:ball},
			$\phi(U(x))=\frac{s^2}{|x-w|^2}$ for $x\notin B$.
			Substitution in the definition of $F$
			and direct differentiation give
			$\Delta(F\circ U)(x)=0$.
			
			Conversely, assume that $\Delta(F\circ U)(x)=0$.
			Then the equality holds in \eqref{eq3.4}.
			Write $m=\lambda(\R^n)$,
			$v=\int_{\R^n}y\,\dd\lambda(y)$,
			$z=\frac{v}{g(m)}$ and $r=|z|\phi(m)^\frac{1}{2}$.
			Since $g(m)>0$ and $\Delta U(x)>0$, 
			\[
			|v|^2
			=
			g(m)\int_{\R^n}|y|^2\,\dd\lambda(y)
			=
			g(m)\frac{\Delta U(x)}{p(p-n+2)}>0.
			\]
			Thus $v\neq o$.
			By Lemma~\ref{lem:moment}, we have
			\[
			\supp\lambda\subseteq B_z(r),
			\qquad
			U_p^\lambda(y)=|y|^{-p}
			\quad\text{for Lebesgue-a.e. }y\in B_z(r).
			\]
		 	By the definition of $\lambda$ and $U_p^\lambda(y)
			=
			|y|^{-p}U\big(x+\frac{y}{|y|^2}\big),$
			it follows that
			\[
			0=\lambda(\R^n\setminus B_z(r))
			=
			\int_{\R^n}
			\mathbf 1_{\R^n\setminus B_z(r)}
			\big(\frac{y-x}{|y-x|^2}\big)
			|y-x|^{-p}\,\dd\sigma(y)
			\]
			and
			\[
			U\big(x+\frac{y}{|y|^2}\big)=1
			\quad\text{for Lebesgue-a.e. }y\in B_z(r).
			\]
			Let $\tau(y)=\frac{y-x}{|y-x|^2}$, for $y\in \mathbb{R}^n\setminus \{x\}$.
			Then $\tau : \mathbb R^n\setminus\{x\} \to \mathbb R^n\setminus\{o\}$ is a diffeomorphism and $\tau^{-1}(y)=x+\frac{y}{|y|^2}$ for $y\in \mathbb{R}^n\setminus \{o\}$.
			Since $m\in(0,1)$ and $0<\phi(m)<1$, it follows that $0<r<|z|$ and therefore $o\notin B_{z}(r)$.
			Let $B = \tau^{-1}(B_z(r))
			=
			B_{x+\frac{z}{|z|^2-r^2}}
			\big(\frac{r}{|z|^2-r^2}\big)$.
			So, $\supp\sigma\subseteq \tau^{-1}(B_z(r))= B$
			and $U(w)=1$ for Lebesgue-a.e. $w\in B$. 
			In the following, we show that $\sigma$ is the
			$p$-capacitary measure of $B$.

			Let $\mu_B$ be the $p$-equilibrium measure of $B$,
			and $\eta_B=\frac{\mu_B}{ V_p(B)}$.
			By Lemma~\ref{lem:ball},
			$U_p^{\eta_B}=1$ on $B$,
			and $\eta_B$ is supported on $B$ and absolutely
			continuous with respect to Lebesgue measure.
			Since $\supp\sigma\subseteq B$,
			$U_p^{\eta_B}=1$  on $B$,
			$U=1$ Lebesgue-a.e. on $B$
			and $\eta_B$ is absolutely continuous,
			by changing the order of integration, we have
			\[
			\sigma(\R^n)
			=
			\int_B U_p^{\eta_B}\dd\sigma
			=
			\int_B U\dd\eta_B
			=
			\eta_B(\R^n)=V_p(B)^{-1}.
			\]
			So, $V_p(B)\sigma\in\mathcal P(B)$.
			By the definition of $V_p(B)$, we have $I_p(V_p(B)\sigma)\ge V_p(B)$.
			
			Meanwhile, by $U\leq1$, we have
			\[
			I_p(V_p(B)\sigma)
			=
			V_p(B)^2\int_B U\dd\sigma
			\leq	V_p(B)^2\sigma(\R^n)=V_p(B).
			\]
			Thus, $I_p(	V_p(B)\sigma)=V_p(B)$,
			and therefore $	V_p(B)\sigma$ is a
			$p$-equilibrium measure of $B$.
			By the uniqueness assertion in
			Lemma~\ref{lem:variation-polar},
			$V_p(B)\sigma=\mu_B$,
			and hence $	\sigma=\eta_B$.
		\end{proof}

	In the following, we prove Theorem \ref{thm:main}. 
	
	\begin{theorem}\label{thm:main1}
	If $n\geq2$, $p\in(n-2,n)$ and $K$ is a nonempty compact set in $\mathbb{R}^n$, then 
		\begin{equation}\label{main ineq 3.3}
			\Cap_p(K)
			\leq
			c_{n,p}\Cap_{n-2}(K)
		\end{equation}
		with equality if and only if  $\Cap_{n-2}(K\setminus B)=0$ for some closed ball $B\subseteq K$.
	\end{theorem}
	
	\begin{proof}
		We first reduce the proof of inequality \eqref{main ineq 3.3} to the case where $\Cap_{n-2}(K)>0$ and  $\Cap_p(K)>0$. Indeed, if $\Cap_p(K)=0$, then \eqref{main ineq 3.3} is immediately obtained.
		Assume $\Cap_p(K)>0$. Then there exists a probability
		measure $\mu\in\mathcal P(K)$ such that $I_p(\mu)<\infty$. 
		By Lemma \ref{lem:riesz-log-limit}, we obtain $I_{n-2}(\mu)<\infty$, and therefore $\Cap_{n-2}(K)>0$. Equivalently,
		\begin{equation}\label{n-2=0=p}
			\Cap_{n-2}(K)=0
			\quad\Longrightarrow\quad
			\Cap_p(K)=0.
		\end{equation}
		Thus, if $\Cap_{n-2}(K)=0$, then \eqref{main ineq 3.3} is immediately obtained. 
		
		In the following, assume $\Cap_{n-2}(K)>0$ and  $\Cap_p(K)>0$.
		
		Let $\Omega_K$ denote the unbounded component of
		$\R^n\setminus K$.
		When $n=2$, let $\mu_0$ denote the $0$-equilibrium measure
		of $K$. 
		When $n\geq3$, let $\eta_{n-2}$ denote the
		$(n-2)$-capacitary measure of $K$. 
		For $x\in\Omega_K$, define
		\[
		W(x)
		=
		\begin{cases}
			F(U(x))+U_{0}^{\mu_0}(x)-V_0(K), & n=2,\\[6pt]
			F(U(x))
			+\frac{U_{n-2}^{\eta_{n-2}}(x)-1}{n-2},
			& n\geq3,
		\end{cases}
		\]
		where $U=U_{p}^{\eta_p}$, $\eta_p$ is the $p$-capacitary measure
		of $K$, and $F$ is defined by \eqref{defofF}.

		\textbf{Step 1.}  We establish the inequality \eqref{main ineq 3.3} by Lemma \ref{lem:exterior-minimum}.
		
		By \eqref{eq:equilibrium-voltage}, we have
		$\eta_p(\R^n)=\Cap_p(K)^p>0$ and $U\leq1$ on $\R^n$.
		Hence, it follows from Lemma~\ref{lem:gradient}  that $U\in(0,1)$ on
		$\Omega_K\subseteq\R^n\setminus\supp\eta_p$, and $\Delta (F\circ U) \leq 0$  on $\Omega_K$. Since $F\geq0$ on $(0,1)$,
		we have $F(U(x))\geq0$ for $x\in\Omega_K$.
		
		\textbf{Case 1.}  Let $n=2$. First, for $x\in\Omega_K$, we have
		\begin{align*}
			\Delta U_0^{\mu_0}(x)
			=
			\int_K
			\Delta_x\big(-\log|x-y|\big)\,\dd\mu_0(y)
			=
			0.
		\end{align*}
	So, $\Delta W\leq 0$ on $\Omega_K$. 
	Second, choose $r_0>1$ such that $K\subseteq B_o(r_0)$. 
	For $x\in B_o(r_0)\cap\Omega_K$,
	\begin{align*}
		W(x) &= (F\circ U)(x) + U_{0}^{\mu_0}(x)-V_0(K)
		\geq
		U_{0}^{\mu_0}(x)-V_0(K)\\
		&=
		\int_K-\log|x-y|\,\dd\mu_0(y)-V_0(K)
		\geq
		-\log(2r_0)-V_0(K).
	\end{align*}
	
	To prove that 	$W$ is bounded below on $\Omega_K$, it suffices to prove that $W$ is bounded below
	on $\Omega_K\setminus B_o(r_0)$.
	By the definition of $\beta,\phi$ and $F$, 
	we have as $t\to 0^+$,
	\begin{align*}
		&\beta(t)
		=c_{2,p}^p t^{\frac{p}{2}}\big(1+O(t)\big),\quad 
		\phi(t)
		=c_{2,p}^{-2}t^{\frac{2}{p}}
		\big(1+O(t^{\frac{2}{p}})\big),\\
			&F(t)
		=-\frac{1}{2}\log\phi(t)
		=-\frac{1}{p}\log t+\log c_{2,p}
		+O(t^{\frac{2}{p}}).
	\end{align*}
		Meanwhile, as $|x|\to\infty$,
	\begin{align*}
		&U(x)
		=\int_K|x-y|^{-p}\,\dd\eta_p(y)=|x|^{-p}\big(1+O(|x|^{-1})\big)\Cap_p(K)^p,\\
		&U_0^{\mu_0}(x)
		=\int_K-\log|x-y|\,\dd\mu_0(y)=-\log|x|+O(|x|^{-1}).
	\end{align*}
	Hence, as $|x|\to\infty$,
		\begin{align}\label{eq:W-infinity-planar}
			W(x)
			&=(F\circ U)(x)+U_0^{\mu_0}(x)-V_0(K)\\
			\nonumber&=-\frac{1}{p}\log U(x)+\log c_{2,p}
			-\log|x|-V_0(K)+O(|x|^{-1})\\
			\nonumber&=\log\big(
			\frac{c_{2,p}\Cap_0(K)}{\Cap_p(K)}
			\big)+O(|x|^{-1}).
		\end{align}
	So, choosing $r_0$ sufficiently big if necessary, we obtain that
	$W$ is bounded below on $\Omega_K$.

Third, by the planar Frostman theorem
\cite[Theorem~1.12]{saff}, $U_0^{\mu_0}\leq V_0(K)$ on $\mathbb{R}^2$.
Combining this inequality with Lemma~\ref{lem:variation-polar}, we obtain
$U_0^{\mu_0}=V_0(K)$, $0$-quasi-everywhere on $K$.
Since $\partial\Omega_K\subseteq K$, it follows that $	U_0^{\mu_0}=V_0(K)$ $0$-quasi-everywhere on $\partial\Omega_K$.
	Thus, from that $F(U(x))\geq0$ on $\Omega_K$ and the lower semicontinuity of $U_0^{\mu_0}$ (\cite[p.~171]{saff}), it follows that for
	$0$-quasi-every $x_0\in\partial\Omega_K$,
	\begin{align*}
		\liminf_{\substack{x\to x_0\\x\in\Omega_K}}W(x)
		&\geq
		\liminf_{\substack{x\to x_0\\x\in\Omega_K}}
		\big(U_0^{\mu_0}(x)-V_0(K)\big)\geq U_0^{\mu_0}(x_0)-V_0(K)=0.
	\end{align*}
	
		Fourth,  by \cite[p.~173]{saff} we have
		\[
		\supp \mu_0\subseteq\partial\Omega_K \quad \text{and } \quad \Cap_0(\partial\Omega_K)=\Cap_0(K)>0.
	\]
	
		Thus, by Lemma \ref{lem:exterior-minimum}, we have
		$W\geq0$ on $\Omega_K$.
		Taking $|x|\to\infty$ in
		\eqref{eq:W-infinity-planar}, we obtain
		\[
		\log\big(
		\frac{c_{2,p}\Cap_0(K)}{\Cap_p(K)}
		\big)\geq0,
		\]
		which yields the inequality \eqref{main ineq 3.3}.
		
		\textbf{Case 2.}  Let $n\ge 3$.  First, for $x\in\Omega_K$, we have
		\begin{align*}
			\Delta U_{n-2}^{\eta_{n-2}}(x)
			=
			\int_K
			\Delta_x\bigl(|x-y|^{2-n}\bigr)\,\dd\eta_{n-2}(y)
			=
			0.
		\end{align*}
		So, $\Delta W\leq 0$ on $\Omega_K$. 

		Second, by \eqref{eq:equilibrium-voltage}, we have
		$U_{n-2}^{\eta_{n-2}}\in(0,1]$ on $\Omega_K$. 
		Moreover, since $U\in(0,1)$ on $\Omega_K$ and $F\geq0$ on $(0,1)$, we have
		\[
		W(x)
		=
		(F\circ U)(x)
		+
		\frac{U_{n-2}^{\eta_{n-2}}(x)-1}{n-2}
		\geq
		-\frac{1}{n-2}
		\qquad\text{on }\Omega_K.
		\]
		So, $W$ is bounded below on $\Omega_K$.
		
	Third, by Fatou's lemma,  $U_{n-2}^{\eta_{n-2}}$ is lower semicontinuous.
	From that $F(U(x))\geq0$ on $\Omega_K$,
	\eqref{eq:equilibrium-voltage} and
	$\partial\Omega_K\subseteq K$, it follows that for
	$(n-2)$-quasi-every $x_0\in\partial\Omega_K$,
	\begin{align*}
		\liminf_{\substack{x\to x_0\\x\in\Omega_K}}W(x)
		&\geq
		\liminf_{\substack{x\to x_0\\x\in\Omega_K}}
		\frac{U_{n-2}^{\eta_{n-2}}(x)-1}{n-2}\geq
		\frac{U_{n-2}^{\eta_{n-2}}(x_0)-1}{n-2}
		=0.
	\end{align*}

Fourth, by the definition of $\beta,\phi$ and $F$, 
we have as $t\to 0^+$,
\begin{align*}
	&\beta(t)
	=c_{n,p}^p t^{\frac{p}{2}}\big(1+O(t)\big),\quad 
	\phi(t)
	=c_{n,p}^{-2}t^{\frac{2}{p}}
	\big(1+O(t^{\frac{2}{p}})\big),\\
	&F(t)=\frac{1}{n-2}-\frac{c_{n,p}^{2-n}}{n-2}\,t^{\frac{n-2}{p}}
	+O(t^{\frac{n}{p}}).
\end{align*}
Meanwhile, as $|x|\to\infty$,
\begin{align*}
		&U(x)
	=\int_K|x-y|^{-p}\,\dd\eta_p(y)=|x|^{-p}\big(1+O(|x|^{-1})\big)\Cap_p(K)^p,\\
	&U_{n-2}^{\eta_{n-2}}(x)
	=
	\int_K|x-y|^{2-n}\,\dd\eta_{n-2}(y)
	=
	\Cap_{n-2}(K)^{n-2}|x|^{2-n}
	\bigl(1+O(|x|^{-1})\bigr).
\end{align*}
Hence
\begin{align}\label{eq:W-infinity-higher}
	W(x)
	&=
	(F\circ U)(x)
	+
	\frac{U_{n-2}^{\eta_{n-2}}(x)-1}{n-2}\\
	&=
	\frac{
		\Cap_{n-2}(K)^{n-2}
		-
		c_{n,p}^{2-n}\Cap_p(K)^{n-2}}{n-2}
	|x|^{2-n}
	+
	O(|x|^{1-n}).\notag
\end{align}
In particular, $W(x)\to0$ as $|x|\to\infty.$ 
		
	Thus, by Lemma \ref{lem:exterior-minimum}, we have
	$W\geq0$ on $\Omega_K$. 	Multiplying \eqref{eq:W-infinity-higher} by $|x|^{n-2}$ and taking
	$|x|\to\infty$, we obtain
	\[
	\Cap_{n-2}(K)^{n-2}
	-
	c_{n,p}^{2-n}\Cap_p(K)^{n-2}
	\geq0.
	\]
	That is, inequality \eqref{main ineq 3.3} holds.

		\textbf{Step 2.} Characterize the equality conditions.
		
		Assume that $B\subseteq K$ is a closed ball and
		$\Cap_{n-2}(K\setminus B)=0$.
		By \eqref{n-2=0=p}, we have $\Cap_p(K\setminus B)=0$.
		By Lemma~\ref{lem:capacity-polar-removal}, we have
		\[
		\Cap_{n-2}(K)=\Cap_{n-2}(B)\quad\text{and}\quad 
		\Cap_p(K)=\Cap_p(B).
		\]
		If $B$ has radius zero, both capacities vanish and equality follows.
		Otherwise, \eqref{eq:ball-capacity-ratio} yields
		\[
		\Cap_p(K)=\Cap_p(B)=c_{n,p}\Cap_{n-2}(B)
		=c_{n,p}\Cap_{n-2}(K).
		\]
		
		Conversely, assume the equality in \eqref{main ineq 3.3} holds.
		
		If $\Cap_{n-2}(K)=\Cap_p(K)=0$, then $\Cap_{n-2}(K\setminus B_y(0))=0$ for each $y\in K$.
		
		In the following, assume $\Cap_{n-2}(K)>0$.
		From
		\eqref{eq:W-infinity-planar} and
		\eqref{eq:W-infinity-higher}, we have
		\[
		W(x)=O(|x|^{1-n})
		\qquad\text{as }|x|\to\infty.
		\]
		
		We first show that $W\equiv0$ on $\Omega_K$.
		From $W\geq0$, the superharmonic mean value inequality
		\cite[Definition~3.1.2]{armitage-gardiner} and the connectedness of $\Omega_K$, it follows that
	  either $W\equiv0$ or $W>0$ throughout $\Omega_K$.
		In the latter case, choose $r_0>0$ such that
		$K\subseteq\operatorname{int}B_o(r_0)$ and put
		$c=\min_{\partial B_o(r_0)}W>0$.
		Then the function $W(x)-c\big(\frac{r_0}{|x|}\big)^{n-2}$
		is superharmonic and bounded below outside $B_o(r_0)$,
		and is nonnegative on $\partial B_o(r_0)$.
		For $n=2$, $\Cap_0(\partial B_o(r_0))=r_0>0$;
		For $n\geq3$, this function tends to $0$ as $|x|\to\infty$.
		Thus, by Lemma~\ref{lem:exterior-minimum}, we have
		\[
		W(x)\geq c\big(\frac{r_0}{|x|}\big)^{n-2}
		\qquad\text{for }|x|>r_0,
		\]
		which contradicts $W(x)=O(|x|^{1-n})$.
		Hence, $W\equiv0$ on $\Omega_K$, and therefore $\Delta(F\circ U)=\Delta W=0$ on $\Omega_K$.
		By the equality condition of Lemma~\ref{lem:gradient},
		there exists a closed ball $B$ with positive radius such that
		$\eta_p=\eta_B$, where $\eta_B$ is the $p$-capacitary measure of $B$. Then $B=\supp\eta_B=\supp\eta_p\subseteq K$.
		
		Furthermore, by the definition of $\eta_p$ and Lemma~\ref{lem:ball}, we have
		\[
		\Cap_p(K)^p
		=\eta_p(\R^n)
		=\eta_B(\R^n)
		=\Cap_p(B)^p.
		\]
		Added with \eqref{eq:ball-capacity-ratio}, we have
		\begin{align*}
			\Cap_{n-2}(K)
			&=\frac{\Cap_p(K)}{c_{n,p}}
			=\frac{\Cap_p(B)}{c_{n,p}}
			=\Cap_{n-2}(B),
		\end{align*}
		and hence $V_{n-2}(K)=V_{n-2}(B)$.
		
		Let $\mu_B$ be the $(n-2)$-equilibrium measure of $B$.
		Since $I_{n-2}(\mu_B)=V_{n-2}(B)=V_{n-2}(K)$, and $\mu_B\in\mathcal P(K)$, it follows that
		$\mu_B$ is also the $(n-2)$-equilibrium measure of $K$.
		By Lemma~\ref{lem:variation-polar}, we have
		\[
		\Cap_{n-2}\big(
		\big\{x\in K:
		U_{n-2}^{\mu_B}(x)<V_{n-2}(K)\big\}
		\big)=0.
		\]
		Meanwhile, by the disk formula
		\cite[Example~1.10]{saff} when $n=2$,
		and radial harmonicity when $n\geq3$, we have
		\[
		U_{n-2}^{\mu_B}<V_{n-2}(B)=V_{n-2}(K)
		\qquad\text{on }\R^n\setminus B.
		\]
		So,
		$K\setminus B\subseteq
		\big\{x\in K:U_{n-2}^{\mu_B}(x)<V_{n-2}(K)\big\}$,
		and hence $\Cap_{n-2}(K\setminus B)=0$.

		This completes the proof.
	\end{proof}

		\vskip15pt
	\section{\bf Proof of Theorem~\ref{thm:newton-log}}
	\label{sec:first-proof}
		\vskip5pt
		
	Throughout this section, let $n\geq4$ and $\nu$ be the normalized surface area measure on $\mathbb S^{n-1}$. Then $\nu$
	 is  also the $(n-2)$-equilibrium measure and  the $0$-equilibrium measure of the unit ball, $I_{n-2}(\nu)=1,$ and $I_0(\nu)=\log c_n$. Please
	  refer to  \cite[Appendix~A]{continuity} for details.
	
	The spherical-shell formula gives
	\begin{equation}\label{eq:spherical-shell-potential}
		U_{n-2}^{\nu}(x)=
		\begin{cases}
			1,&|x|\leq1,\\
			|x|^{2-n},&|x|>1.
		\end{cases}
	\end{equation}
	By the rotational invariance of $\nu$, we have
	\[
	U_0^{\nu}(x)=I_0(\nu)=\log c_n
	\qquad\text{for }|x|=1.
	\]
	Since $\frac{\dd}{\dd r}U_0^{\nu}(re)
	=-\int_{\mathbb S^{n-1}}
	\frac{r-\theta\cdot e}{r^2+1-2r\theta\cdot e}\,\dd\nu(\theta)<0$ for all $r>1$ and $e\in\mathbb S^{n-1}$, we have
	\begin{equation}\label{eq:log-shell-strict}
		U_0^{\nu}(x)<\log c_n\qquad\text{for }|x|>1.
	\end{equation}
	
	The \emph{classical Riesz composition formula} (See, e.g., \cite{landkof}) reads:
		If $0<p,q<n$ and $p+q>n$, then for 
		$y,z\in\R^n$ with $y\neq z$,
		\begin{equation}
			\label{lem:riesz-composition}
			\int_{\R^n}\frac{\dd x}{|x-y|^p|x-z|^q}
			=
			\pi^{\frac n2}
			\frac{
				\Gamma(\frac{n-p}{2})
				\Gamma(\frac{n-q}{2})
				\Gamma(\frac{p+q-n}{2})
			}{
				\Gamma(\frac p2)
				\Gamma(\frac q2)
				\Gamma(\frac{2n-p-q}{2})
			}|y-z|^{n-p-q}.
		\end{equation}

	For a nonnegative measurable function $f$ satisfying $|\{f>t\}|<\infty$ for $t>0$,
	its symmetric decreasing rearrangement is
	\[
	f^*(x)
	=\inf\big\{t\geq0:|\{f>t\}|\leq\omega_n|x|^n\big\},
	\qquad x\in\R^n.
	\]
         For  such functions $f,g,h$,
     the classical \emph{Riesz rearrangement inequality}
(See, e.g., \cite[Theorem~3.7]{lieb-loss}) reads:
	\begin{align*}
		\int_{\mathbb{R}^n}\int_{\mathbb{R}^n}
		f(x)g(x-y)h(y)\,dx\,dy
		\le
		\int_{\mathbb{R}^n}\int_{\mathbb{R}^n}
		f^*(x)g^*(x-y)h^*(y)\,dx\,dy.
	\end{align*}

The classical \emph{P\'olya--Szeg\H{o} inequality}
	(See, e.g., \cite[Eq.~(1.1) and Theorem~1.1]{burchard-ferone}) reads:
	If $f\in L^{\frac{2n}{n-2}}(\R^n)$ with $f\ge 0$, and its weak gradient $\nabla f\in L^2(\R^n)$,
	then $\nabla f^*\in L^2(\R^n)$ and
	\begin{equation}\label{eq:polya-szego}
		\int_{\R^n}|\nabla f^*|^2\,\dd x
		\leq\int_{\R^n}|\nabla f|^2\,\dd x.
	\end{equation}
	If, in addition, $\big|\{x\in\R^n:0<f^*(x)<\|f\|_{L^\infty},|\nabla f^*(x)|=0\}\big|=0,$
	then equality holds in \eqref{eq:polya-szego} if and only if $	f(x)=f^*(x-y)$ Lebesgue-a.e. on $\R^n$ for some $y\in \R^n$.
	
	The  lemma below follows from the Talenti comparison theorem
	\cite[Theorem~1]{talenti} by approximation. We give a 
	capacitary proof and characterize the equality condition.
	\begin{lemma}\label{lem:distribution}
		If $\mu$ is a compactly supported probability measure
		and $U_{n-2}^{\mu}\leq1$ on $\mathbb{R}^n$, then
		$	\big(U_{n-2}^{\mu}\big)^*\leq U_{n-2}^{\nu}$ Lebesgue-a.e. on $\R^n$.
		Moreover, 
		$(U_{n-2}^{\mu})^*=U_{n-2}^{\nu}$ Lebesgue-a.e. on $\R^n$,
		iff $\mu$ is the normalized surface measure on a sphere of radius $1$.
	\end{lemma}
	
	\begin{proof}
		We first prove that  $\big(U_{n-2}^{\mu}\big)^*\leq U_{n-2}^{\nu}$
		almost everywhere on $\R^n$.
		
		Fix $0<t<1$, and let $F\subseteq \{U_{n-2}^{\mu}>t\}$ be a 
		compact set. If $\Cap_{n-2}(F)>0$, let $\eta_F$
		be the $(n-2)$-capacitary measure of $F$. 
		By \eqref{eq:equilibrium-voltage}, we have
		\[
		\eta_F(\R^n)=\Cap_{n-2}(F)^{n-2},
		\qquad U_{n-2}^{\eta_F}\leq1\quad\text{on }\R^n.
		\]
		Thus, by the Tonelli theorem,
		\begin{align*}
			t\Cap_{n-2}(F)^{n-2}
			=t\eta_F(\R^n)\leq\int_{\R^n} U_{n-2}^{\mu}\,\dd\eta_F
			=\int_{\R^n} U_{n-2}^{\eta_F}\,\dd\mu\leq\mu(\R^n)=1.
		\end{align*}
		So, $\Cap_{n-2}(F)\leq t^{-\frac{1}{n-2}}$.
		By the classical isocapacitary inequality,  we have
		\begin{align}\label{4.1}
			|F|\leq\omega_n\Cap_{n-2}(F)^n
			\leq\omega_n t^{-\frac{n}{n-2}},
		\end{align}
		which also holds if $\Cap_{n-2}(F)=0$.
		Taking the supremum over compact subsets $F$ of $\{U_{n-2}^{\mu}>t\}$ in \eqref{4.1}  and by \eqref{eq:spherical-shell-potential}, we obtain
		\[
		|\{U_{n-2}^{\mu}>t\}|\leq\omega_n t^{-\frac{n}{n-2}}
		=|\{U_{n-2}^{\nu}>t\}|,\qquad 0<t<1.
		\]
		For $t\geq1$, $|\{U_{n-2}^{\mu}>t\}|=0=|\{U_{n-2}^{\nu}>t\}|$. 
		Hence, $\big(U_{n-2}^{\mu}\big)^*\leq U_{n-2}^{\nu}$
		Lebesgue-a.e. on $\R^n$.
		
		In the following, we prove the equality condition.
		
		If $\mu$ is the normalized surface measure on
		$\partial B_y(1)$ for some $y\in\R^n$, then
		$U_{n-2}^{\mu}(x)=U_{n-2}^{\nu}(x-y)$. By the definition of  symmetric decreasing rearrangement, we have
		$\big(U_{n-2}^{\mu}\big)^*=U_{n-2}^{\nu}$ 	Lebesgue-a.e. on $\R^n$.
		
			Conversely, suppose  that $\big(U_{n-2}^{\mu}\big)^*=U_{n-2}^{\nu}$ 	Lebesgue-a.e. on $\R^n$. Since
		$U_{n-2}^{\mu}\leq1$ and $\mu$ is a probability measure, we have
		\[
		I_{n-2}(\mu)=\int_{\R^n} U_{n-2}^{\mu}\,\dd\mu\leq1<\infty.
		\]
		By the Newtonian energy identity \cite[Theorem 1.20]{landkof}, we have		$\nabla U_{n-2}^{\mu}\in L^2(\R^n)$ and
		\[
		\int_{\R^n}|\nabla U_{n-2}^{\mu}|^2\,\dd x
		=n(n-2)\omega_n I_{n-2}(\mu).
		\]
		Since $\mu$ is compactly supported,
		$U_{n-2}^{\mu}(x)=O(|x|^{2-n})$ as $|x|\to\infty$.
		Combining this with $0<U_{n-2}^{\mu}\leq1$,  we have for sufficiently
		large $R_0$ and some $C>0$,
		\[
		\int_{\R^n}\big(U_{n-2}^{\mu}\big)^{\frac{2n}{n-2}}\,\dd x
		\leq |B_o(R_0)|+C\int_{R_0}^{\infty}r^{-n-1}\,\dd r<\infty.
		\]
		Therefore, $U_{n-2}^{\mu}$ vanishes at infinity, $U_{n-2}^{\mu}\in L^{\frac{2n}{n-2}}(\R^n)$, and
		$\nabla U_{n-2}^{\mu}\in L^2(\R^n)$.
		
		By $I_{n-2}(\nu)=1$ and the P\'olya--Szeg\H{o} inequality \eqref{eq:polya-szego}, we have
		\[
		\begin{aligned}
			&n(n-2)\omega_n
			=\int_{\R^n}|\nabla U_{n-2}^{\nu}|^2\,\dd x
			=\int_{\R^n}|\nabla \big(U_{n-2}^{\mu}\big)^*|^2\,\dd x\\
			\leq&\int_{\R^n}|\nabla U_{n-2}^{\mu}|^2\,\dd x
			=n(n-2)\omega_n I_{n-2}(\mu)\leq n(n-2)\omega_n.
		\end{aligned}
		\]
		Hence, all the equalities hold.
		Since $\big(U_{n-2}^{\mu}\big)^*=U_{n-2}^{\nu}$
		a.e. on $\R^n$, it follows that
		\[
			\|U_{n-2}^{\mu}\|_{L^\infty}
		=\|\big(U_{n-2}^{\mu}\big)^*\|_{L^\infty}=1.
		\]
		Moreover, by \eqref{eq:spherical-shell-potential} we have 	$|\nabla U_{n-2}^{\nu}(x)|
		=(n-2)|x|^{1-n}>0$ for all $x\in \mathbb{R}^n\setminus B_{o}(1)$.
		So, 
		\[
		\big|\{x\in\R^n:0<\big(U_{n-2}^{\mu}\big)^*(x)<1,
		\ |\nabla\big(U_{n-2}^{\mu}\big)^*(x)|=0\}\big|=0.
		\]
	By the equality condition of \eqref{eq:polya-szego}, there
		exists $y\in\R^n$ such that $U_{n-2}^{\mu}(x)=U_{n-2}^{\nu}(x-y)$ Lebesgue-a.e. on $\R^n$.
		Hence, for all $\varphi\in C_c^\infty(\R^n)$,
		\begin{align*}
			&\int_{\R^n}\varphi\,\mathrm d\mu={}-\frac{1}{n(n-2)\omega_n}\int_{\R^n}U_{n-2}^{\mu}(x)\Delta\varphi(x)\,\mathrm dx\\
			={}&-\frac{1}{n(n-2)\omega_n}\int_{\R^n}U_{n-2}^{\nu}(x-y)\Delta\varphi(x)\,\mathrm dx=\int_{\R^n}\varphi(x+y)\,\mathrm d\nu
		\end{align*}
		Thus $\mu$ is the translate of $\nu$ by $y$, namely the normalized
		surface measure on $\partial B_y(1)$.
		
	This completes the proof.
	\end{proof}

	We now prove Theorem~\ref{thm:newton-log}.
	
	\begin{theorem}\label{thm:newton-log4.1}
		If $n\geq4$ and $K$ is a nonempty compact set in $\mathbb{R}^n$, then
		\begin{equation}\label{eq:newton-log4.1}
			\Cap_{n-2}(K)\leq c_n\Cap_0(K)
		\end{equation}
		with equality if and only if $\Cap_0(K\setminus B)=0$ and $\partial B\subseteq K$ for some closed ball $B$.
	\end{theorem}
	
	\begin{proof}
		We first reduce the proof of inequality \eqref{eq:newton-log4.1} to the case where $\Cap_{n-2}(K)>0$ and $\Cap_{0}(K)>0$.
		Indeed, if $\Cap_{n-2}(K)=0$, then \eqref{eq:newton-log4.1} is immediately obtained.
		Assume $\Cap_{n-2}(K)>0$. Then there exists  $\mu\in\mathcal P(K)$ such that $I_{n-2}(\mu)<\infty$. 
		By Lemma \ref{lem:riesz-log-limit}, we obtain $I_0(\mu)<\infty$, and therefore $\Cap_{0}(K)>0$.
		Equivalently,
		\begin{equation}\label{eq:log-newton-polar}
			\Cap_0(K)=0\quad\Longrightarrow\quad\Cap_{n-2}(K)=0.
		\end{equation}
		Thus, if $\Cap_{0}(K)=0$ or $\Cap_{n-2}(K) = 0$, then \eqref{eq:newton-log4.1} is immediately obtained. 
		
		In the following, w.l.o.g., assume  $\Cap_{0}(K)>0$ and  $\Cap_{n-2}(K)=1$.

		Let $\mu_{n-2}$ be the $(n-2)$-equilibrium measure of $K$, and write
		$U=U_{n-2}^{\mu_{n-2}}$. From 
		$\Cap_{n-2}(K)=1$, together with \eqref{Lem 2.1 cor}, we have $I_{n-2}(\mu_{n-2})=1,$ and $0<U\leq1.$
		
		\smallskip
		\textbf{Step 1.} Prove inequality \eqref{eq:newton-log4.1}. It suffices to prove that $I_0(\mu_{n-2})\leq I_0(\nu)$. If so, by  $V_0(K)\leq I_0(\mu_{n-2})$ and $I_0(\nu)=\log c_n$, it follows that
		\begin{equation}\label{eq:newton-log-energy-chain}
			\Cap_0(K)=e^{-V_0(K)}\geq e^{-I_0(\mu_{n-2}) }\geq e^{-I_0(\nu) }    =\frac{1}{c_n} 
			=\frac{\Cap_{n-2}(K)}{c_n}.
		\end{equation}
		
		\textbf{Case 1.} Let $n\geq5$.
		For $0<p<n-4$, from Lemma \ref{lem:riesz-log-limit}, and applying \eqref{lem:riesz-composition}
		with the exponent pairs $(n-2,p+2)$ and $(n-2,p+4)$ twice, we have
		\[
		\begin{aligned}
		\infty >	I_p(\mu_{n-2})
			&=\frac{p(n-2-p)}{n(n-2)\omega_n}
			\int_{\R^n}\int_{\R^n}
			\frac{U(x)}{|x-y|^{p+2}}\,\dd x\dd\mu_{n-2}(y) \\
			&=A_{n,p}\int_{\R^n}\int_{\R^n}
			\frac{U(x)U(y)}{|x-y|^{p+4}}\,\dd x\dd y,
		\end{aligned}
		\]
		where $A_{n,p}
		=\frac{p(n-2-p)(p+2)(n-4-p)}{n^2(n-2)^2\omega_n^2}>0.$ 
		Similarly, we have $$\infty >I_p(\nu)
		=A_{n,p}\int_{\R^n}\int_{\R^n}
		\frac{U_{n-2}^{\nu}(x)U_{n-2}^{\nu}(y)}{|x-y|^{p+4}}\,\dd x\dd y.$$
		
		By the Riesz rearrangement inequality  with $f=h=U$ and
		$g(x)=|x|^{-(p+4)}$, we have
		\[
		\int_{\R^n}\int_{\R^n}\frac{U(x)U(y)}{|x-y|^{p+4}}\,\dd x\dd y
		\leq
		\int_{\R^n}\int_{\R^n}\frac{U^*(x)U^*(y)}{|x-y|^{p+4}}\,\dd x\dd y.
		\]
		Therefore,
		\begin{equation*}
			\begin{split}
				I_p(\nu)-I_p(\mu_{n-2})
				\geq A_{n,p}\int_{\R^n}\int_{\R^n}
				\frac{U_{n-2}^{\nu}(x)U_{n-2}^{\nu}(y)-U^*(x)U^*(y)}{|x-y|^{p+4}}\,\dd x\dd y.
			\end{split}
		\end{equation*}
		From Lemma~\ref{lem:riesz-log-limit}, the Fatou Lemma together with Lemma \ref{lem:distribution} and $\lim\limits_{p\to 0^{+}}\frac{A_{n,p}}{p}
		=\frac{2(n-4)}{n^2(n-2)\omega_n^2}$, it follows that
		\begin{align*}
				&\quad\ \  I_0(\nu)-I_0(\mu_{n-2})
			=\lim_{p\to0^+}
			\frac{I_p(\nu)-I_p(\mu_{n-2})}{p}\\
			&\geq
			\int_{\R^n}\int_{\R^n}
			\liminf_{p\to0^+}
			\big(
			\frac{A_{n,p}}{p}
			\frac{
				U_{n-2}^{\nu}(x)U_{n-2}^{\nu}(y)-U^*(x)U^*(y)
			}{|x-y|^{p+4}}
			\big)\,\dd x\,\dd y\\
			&=
			\frac{2(n-4)}{n^2(n-2)\omega_n^2}
			\int_{\R^n}\int_{\R^n}
			\frac{
				U_{n-2}^{\nu}(x)U_{n-2}^{\nu}(y)-U^*(x)U^*(y)
			}{|x-y|^4}\,\dd x\,\dd y\geq0.
		\end{align*}
	That is,
		\begin{equation}\label{eq:defect-high}
			I_0(\nu)-I_0(\mu_{n-2})
			\geq0.
		\end{equation}

		\smallskip
		\textbf{Case 2.} Let $n=4$.
		Fix $r>1$ such that $K\subseteq B_o(r)$.
		
		First, as $R\to +\infty$, we have
		\begin{align}
			&\ \int_{B_o(R)}U^2\,\dd x
			=\int_{\R^4}\int_{\R^4}
			\int_{B_o(R)}\frac{\dd x}{|x-y|^2|x-z|^2}
			\dd\mu_2(y)\dd\mu_2(z)    \label{eq:four-u} \\
			=&\ \int_{\R^4}\int_{\R^4}
			\int_{B_{-y}(R)}\frac{\dd x}{|x|^2|x-(z-y)|^2}
			\dd\mu_2(y)\dd\mu_2(z)   \notag\\
			=&\ \int_{\R^4}\int_{\R^4}
			\int_{B_o(R)}\frac{\dd x}{|x|^2|x-(z-y)|^2}
			\dd\mu_2(y)\dd\mu_2(z)+O(R^{-1}) \notag\\
			=&\ \int_{\R^4}\int_{\R^4}\int_0^R\int_{\mathbb S^3}
			\frac{\rho^3\,\dd\mathcal H^3(\theta)}{\rho^2|\rho\theta-(z-y)|^2} \dd\rho
			\dd\mu_2(y)\dd\mu_2(z)+O(R^{-1}) \notag\\
			=&\ 2\pi^2\int_{\R^4}\int_{\R^4}
			\int_0^R\frac{\rho}{\max\{\rho,|y-z|\}^2}\,\dd\rho
			\dd\mu_2(y)\dd\mu_2(z)+O(R^{-1}) \notag\\
			=&\ 2\pi^2\int_{\R^4}\int_{\R^4}\big(
			\int_0^{|y-z|}\frac{\rho}{|y-z|^2}\,\dd\rho+\int_{|y-z|}^R\frac{\dd\rho}{\rho}
			\big)\dd\mu_2(y)\dd\mu_2(z)+O(R^{-1}) \notag\\
			=&\ 2\pi^2\int_{\R^4}\int_{\R^4}\big(
			\frac12+\log R-\log|y-z|
			\big)\dd\mu_2(y)\dd\mu_2(z)+O(R^{-1}) \notag\\
			=&\ \pi^2+2\pi^2\log R+2\pi^2I_0(\mu_2)+O(R^{-1}). \notag
		\end{align}
		Here the first equality follows from the Tonelli theorem;
		The third equality follows from the facts that for sufficiently large $R$, the symmetric difference $B_{-y}(R)\mathbin{\triangle}B_o(R)
		\subseteq B_o(R+r)\setminus B_o(R-r)$, the volume $|B_o(R+r)\setminus B_o(R-r)| = O(R^3)$ and  uniformly for \(y,z\in K\)
		\begin{align*}
			\frac{1}{|x|^2|x-(z-y)|^2} = O(R^{-4})\quad \text{ for }\   x\in B_o(R+r)\setminus B_o(R-r);
		\end{align*}
		The fifth equality follows from
		\eqref{eq:spherical-shell-potential}, which gives
		\[
		\int_{\mathbb S^3}
		\frac{\dd\mathcal H^3(\theta)}{|\rho\theta-(z-y)|^2}
		=\frac{2\pi^2}{\rho^2}U_2^{\nu}\!\big(\frac{z-y}{\rho}\big)
		=\frac{2\pi^2}{\max\{\rho,|y-z|\}^2}.
		\]
		
		Second, by \eqref{eq:spherical-shell-potential}, we have
		\begin{equation}\label{eq:four-b}
			\begin{aligned}
				\int_{B_o(R)}\big(U_{n-2}^{\nu}\big)^2\,\dd x
				=2\pi^2\big(\int_0^1 \rho^3\,\dd\rho+\int_1^R\frac{\dd\rho}{\rho}\big)
				=2\pi^2\log R+\frac{\pi^2}{2},\qquad R>1.
			\end{aligned}
		\end{equation}
		
		Third, write $E_t=\{x\in\R^4:U(x)>t\}$. Then
		\begin{align*}
			\ &\ \int_{B_o(R)}((U^*)^2-U^2)\,\dd x
			=\int_0^\infty2t\big(
			\min\{\omega_4R^4,|E_t|\}-|B_o(R)\cap E_t|\big)\,\dd t   \\
			=&\ \int_{(R+r)^{-2}}^{(R-r)^{-2}}2t\big(
			\min\{\omega_4R^4,|E_t|\}-|B_o(R)\cap E_t|\big)\,\dd t \notag\\
			\leq&\ \omega_4R^4\int_{(R+r)^{-2}}^{(R-r)^{-2}}2t\,\dd t
			=\omega_4R^4\big((R-r)^{-4}-(R+r)^{-4}\big)
			=O(R^{-1}). \notag
		\end{align*}
		Here the first equality follows from  that $|B_o(R)\cap\{U^*>t\}|=\min\{\omega_4R^4,|E_t|\}$ and the coarea formula;
		Since $U(x)=\int_{\R^4}|x-y|^{-2}\,\dd\mu_2(y)$, the triangle inequality gives that
		\[
		U\geq(R+r)^{-2}\quad\text{on }B_o(R),
		\quad\text{and}\quad U\leq(R-r)^{-2}\quad\text{on }\R^4\setminus B_o(R).
		\]
		Hence $B_o(R)\subseteq E_t$ for $0<t<(R+r)^{-2}$, whereas
		$E_t\subseteq B_o(R)$ for $t>(R-r)^{-2}$. 
		This gives the second equality.
		
		Moreover, from the first equality, we conclude that
		\begin{align*}
			0\leq \int_{B_o(R)}((U^*)^2-U^2)\,\dd x \leq O(R^{-1}).
		\end{align*}
		So, combining Lemma \ref{lem:distribution},  \eqref{eq:four-u}, \eqref{eq:four-b} and 
		$I_0(\nu)=\log c_4=-\frac{1}{4}$, we have
		\begin{align}
			0\leq & \lim_{R\to\infty}\int_{B_o(R)}\big(\big(U_{n-2}^{\nu}\big)^2-(U^*)^2\big)\,\dd x  \label{case1 step2eq} \\
			= &\lim_{R\to\infty}\int_{B_o(R)}\big( (U_{n-2}^{\nu})^2-U^2+ U^2 -(U^*)^2\big)\,\dd x \notag\\ 
			= &-\frac{\pi^2}{2}  - 2\pi^2I_0(\mu_2)
			=2\pi^2\big(I_0(\nu)-I_0(\mu_2)\big). \notag
		\end{align}
		Thus, $I_0(\mu_2) \leq I_0(\nu)$.

		\textbf{Step 2.} Characterize the equality conditions.
		
		Assume that there exists some closed ball $B$ such that
		$\partial B\subseteq K$ and $\Cap_0(K\setminus B)=0$. By \eqref{eq:log-newton-polar}, we have $\Cap_{n-2}(K\setminus B)=0$.
		So, we have $\Cap_0(K)=\Cap_0(K\cap B)$ and $\Cap_{n-2}(K)=\Cap_{n-2}(K\cap B)$ by
		Lemma~\ref{lem:capacity-polar-removal}.
		
		Since $\partial B\subseteq K\cap B\subseteq B$, it follows that
		\[	\Cap_{n-2}( B)=\Cap_{n-2}(\partial B)\le \Cap_{n-2}(K\cap B)\le\Cap_{n-2}(B),\]
		\[\Cap_{0}( B)=\Cap_{0}(\partial B)\le \Cap_{0}(K\cap B)\le\Cap_{0}(B).\]
		Therefore, 
		\begin{align*}
			&\Cap_{n-2}(K)=\Cap_{n-2}(B)
			=c_n\Cap_{0}(B)=c_n\Cap_0(K).
		\end{align*}
		
		Conversely, 
		if $\Cap_{n-2}(K)=c_n\Cap_0(K)=0$, then for any $y\in K$, the
		ball $B=B_y(0)$ satisfies $\partial B\subseteq K$ and
		$\Cap_0(K\setminus B)=0$.
		
		Assume that $\Cap_{n-2}(K)=c_n\Cap_0(K)>0$. 
		By \eqref{eq:newton-log-energy-chain}, we have
		\begin{align*}
			I_0(\mu_{n-2})=I_0(\nu)=V_0(K).
		\end{align*}
		So, $\mu_{n-2}$ is also
		the $0$-equilibrium measure of $K$. By Lemma~\ref{lem:variation-polar}, we have
		\begin{equation}\label{final}
			\Cap_0(\{x\in K:U_0^{\mu_{n-2}}(x)< V_0(K)\})=0.
		\end{equation}

		If $n\geq5$, from \eqref{eq:defect-high} together with that
		\begin{align*}
			U_{n-2}^{\nu}(x)U_{n-2}^{\nu}(y)-U^*(x)U^*(y) \geq 
			\big(U_{n-2}^{\nu}(x)-U^*(x)\big)U_{n-2}^{\nu}(y)\geq0,
		\end{align*}
		it follows that $U^*=U_{n-2}^{\nu}$ Lebesgue-a.e. on $\R^n$.
		If $n=4$, from \eqref{case1 step2eq} together with Lemma~\ref{lem:distribution}, we have $U^*=U_2^{\nu}$ Lebesgue-a.e. on $\R^4$.
		
		Thus, by the equality condition of Lemma~\ref{lem:distribution},
		$\mu_{n-2}$ is the normalized surface measure on $\partial B$
		for some ball $B$ of radius $1$. Since $\mu_{n-2}\in\mathcal P(K)$, we have $\partial B=\supp\mu_{n-2}\subseteq K.$
		Moreover, by \eqref{eq:log-shell-strict} and the translation invariance, we have
		\[
		U_0^{\mu_{n-2}}(x)<\log c_n=V_0(K)
		\qquad\text{for }x\notin B.
		\]
		Combining this with \eqref{final}, we have $\Cap_0(K\setminus B)=0$.
		This completes the proof.
	\end{proof}
	\vskip10pt
	\noindent\textbf{AI Use Statement.}
	Through an iterative dialogue guided by our prompts,
	GPT-6.0 helped us identify the auxiliary function $F$
	defined in \eqref{defofF}, which we use to overcome the difficulty raised by
	the nonharmonicity of Riesz $p$-potentials.
	
	The authors take full responsibility for the content of this paper.

\end{document}